\documentclass[a4paper,12pt]{article}
\newcommand{\toukou}[1]{\ifx\TOUKOU\undefined\else{#1}\fi}%
\newcommand{\toukoudel}[1]{\ifx\TOUKOU\undefined{#1}\else\fi}%
\newcommand{\toukouchange}[2]{\ifx\TOUKOU\undefined{#1}\else{#2}\fi}%
\toukoudel{
\setlength{\oddsidemargin}{2mm}
\setlength{\evensidemargin}{2mm}
\setlength{\topmargin}{-8mm}
\setlength{\textwidth}{156mm}
\setlength{\textheight}{230mm}

}
\usepackage{amssymb}
\usepackage{amsmath}
\usepackage[dvipdfmx]{graphicx}
\usepackage{verbatim,enumerate}
\usepackage[active]{srcltx}
\usepackage[normalem]{ulem} 
\usepackage[noadjust]{cite}
\usepackage{times}
\usepackage{amsthm}
\newtheorem{theorem}{Theorem}[section]
\newtheorem{proposition}[theorem]{Proposition}
\newtheorem{fact*}{Fact}
\newtheorem{fact}[theorem]{Fact}

\newtheorem{corollary}[theorem]{Corollary}
\theoremstyle{definition}
\newtheorem{definition}[theorem]{Definition}

\newtheorem*{remark*}{Remark}

\newtheorem{example}[theorem]{Example}
\numberwithin{equation}{section}

\usepackage[usenames]{color}

\newcommand{\R}{\boldsymbol{R}}

\newcommand{\rank}{\operatorname{rank}}

\newcommand{\hess}{\operatorname{Hess}}

\renewcommand{\phi}{\varphi}
\newcommand{\sgn}{\operatorname{sgn}}
\newcommand{\inner}[2]{\left\langle{#1},{#2}\right\rangle}

\newcommand{\A}{\mathcal{A}}

\newcommand{\pmt}[1]{{\begin{pmatrix} #1  \end{pmatrix}}}

\newcommand{\mycomment}[1]{}
\makeatletter
  \newcommand{\subsubsubsection}{\@startsection{paragraph}{4}{\z@}%
    {-1ex \@plus -1ex \@minus -.2ex}%
    {1.0ex \@plus.2ex}
    {\reset@font\bfseries\normalsize}
  }
  \makeatother
\begin{document}
\toukouchange{
\begin{center}
{\large {\bf
Singularities of
a surface given by Kenmotsu-type formula in Euclidean three-space
}}
\\[2mm]
\today
\\[2mm]
\renewcommand{\thefootnote}{\fnsymbol{footnote}}
Luciana F. Martins,
Kentaro Saji
 and
Keisuke Teramoto
\end{center}
}
{
\title{Singularities of
a surface given by Kenmotsu-type formula in Euclidean three-space
\thanks{Partly supported by the
JSPS KAKENHI Grant Numbers 18K03301 and 17J02151, and by grant 2018/19610-7, S\~ao Paulo Research Foundation (FAPESP).
This work was partially done during the first author's
stay at Department of Mathematics, Kobe University 2018,
by the support of the JSPS KAKENHI Grant Number 17H06127.
The authors would like to thank these supports.}
}
\titlerunning{Singularities of a surface given by Kenmotsu-type formula}
\author{Luciana F. Martins \and Kentaro Saji \and Keisuke Teramoto}
\authorrunning{L. Martins, K. Saji and K. Teramoto}
\institute{L. F. Martins \at
Departamento de Matem{\'a}tica,\\
Universidade Estadual Paulista (Unesp),\\
Instituto de Bioci\^{e}ncias, 
Letras e Ci\^{e}ncias Exatas,\\
 S\~{a}o Jos\'{e} do Rio Preto, \\
SP, Brazil\\
E-mail: {\tt luciana.martinsO\!\!\!aunesp.br}\\
\and
K. Saji \at
Department of Mathematics,\\
Graduate School of Science, \\
Kobe University, \\
Rokkodai 1-1, Nada, Kobe \\
657-8501, Japan\\
  E-mail: {\tt sajiO\!\!\!amath.kobe-u.ac.jp}\\
\and 
K. Teramoto \at
Institute of Mathematics for Industry, \\
Kyushu University, \\
Motooka 744, Nishi-ku, Fukuoka \\
819-0395, Japan\\
E-mail: {\tt k-teramotoO\!\!\!aimi.kyushu-u.ac.jp}\\
}
\date{Received: date / Accepted: date}
\maketitle
}

\begin{abstract}
We study singularities of surfaces
which are given by Kenmotsu-type formula
with prescribed unbounded mean curvature.
\toukou{
}
\end{abstract}
\section{Introduction}
In \cite{k3}, Kenmotsu gave a formula
which describes 
immersed surfaces in the Euclidean $3$-space
by prescribed mean curvature and Gauss map.
Furthermore, it is generalized to the
Lorentz-Minkowski $3$-space (\cite{an,m}).
These formulas are considered important
in Euclidean or Lorentzian surface theory
since  they are similar to the  Weierstrass representation for minimal surfaces. 
A coordinate free formula of these formulas is
obtained by Kokubu (\cite{koku}), unifying them in a single one, 
which is called {\it Kenmotsu-type formula}.

On the other hand, in the recent decades, there are several articles
concerning differential geometry of singular curves and surfaces,
namely, curves and surfaces with singular points, in the $2$ and $3$
dimensional Euclidean spaces
\cite{bw,fh,ft,hnuy,hnuy2,irrf,MSUY,nuy,OT,front}. Especially,
wave fronts (fronts) is a class of surfaces with
singularities, where the unit normal vector is well-defined even on
the set of singular points. If a surface $f$ has a singularity, the
mean curvature $H$ may diverge. In the case that $f$ is a front,
behaviors, in particular boundedness of $H$ near singular points
are investigated in \cite{MSUY} and \cite{front}.

It is natural to expect that there is a formula
which describes singular surfaces
by unbounded mean curvature and Gauss map.
In this work, we construct
a Kenmotsu-type formula for singular surfaces
of prescribed
unbounded mean curvature and Gauss map
by a little modification of Kokubu's formula 
(see Theorem \ref{thm:construction}).
Furthermore, we study singularities
of these surfaces and their geometric invariants.


\section{Fronts, their geometric invariants and Kossowski metric}


\subsection{Fronts and their mean curvatures}

A map-germ $f:(\R^2,p)\to(\R^3,0)$ at $p$ is a called {\it frontal\/}
if there exists a map (called the {\it unit normal vector field\/})
$\nu:(\R^2,p)\to(\R^3,0)$
satisfying $|\nu|=1$ and $\inner{df(X)}{\nu}=0$ holds identically,
where $\inner{~}{~}$ is the Euclidean inner product on $\R^3$.
A frontal is a {\it front\/} if the pair $(f,\nu)$ is
an immersion.
We call the function
$$
\det(f_u, f_v,\nu)
$$
the {\it signed area density function},
and a function is called {\it singularity identifier\/}
if it is a non-zero functional multiple of the signed area density function.
A {\it cuspidal edge\/} is a map-germ $(\R^2,p)\to(\R^3,0)$ at $p$
which is ${\cal A}$-equivalent to the map-germ
$(u,v)\mapsto(u,v^2,v^3)$ at $(0,0)$.
A map-germ $(\R^2,p)\to(\R^3,0)$
which is ${\cal A}$-equivalent to
$(u,v)\mapsto(u,4v^3+2uv,3v^2+uv^2)$
(respectively, $(u,v)\mapsto(u,5v^4+2uv,4v^5+uv^2-u^2)$)
is called a {\it swallowtail}
(respectively, a {\it cuspidal butterfly\/}).
Furthermore, a map-germ which is $\A$-equivalent to
$(u,v)\mapsto(u,-2v^3+u^2v,3v^4-u^2v^2)$
(respectively $(u,v)\mapsto(u,-2v^3-u^2v,3v^4-u^2v^2)$)
is called a {\it cuspidal lips\/} (respectively, {\it cuspidal beaks\/}).
These map-germs are fronts
and it is known that the generic singularities of fronts
are cuspidal edges and swallowtails.
Furthermore, cuspidal butterflies, cuspidal lips and beaks
in addition above two are
the generic singularities of one-parameter 
families of fronts (\cite{AGV}).

Let $f:(\R^2,p)\to(\R^3,0)$ be a front-germ and
$\nu$ the unit normal vector field.
The set of singular points of $f$ is denoted by $S(f)$.
A singular point $p$ of $f$ is said to be of {\it corank one\/}
if $\rank df_p=1$.
Let $p$ be a corank one singular point of a front $f$.
Then there exists a non-zero vector field $\eta$
near $p$ such that $\eta(q)$ generates
the kernel of $df_q$ for any $q\in S(f)$.
We call $\eta$ the {\it null vector field}.
A singular point $p$ of $f$ is
{\it non-degenerate\/} if $d\lambda_p\ne0$,
where $\lambda$ is a singularity identifier.
Otherwise, it is {\it degenerate}.
Let $p$ be a non-degenerate singular point.
Then there exists a regular curve-germ
$\gamma:(\R,0)\to (\R^2,p)$ such that
the image of $\gamma$ coincides with $S(f)$.
One can easily see that a non-degenerate singular point
is corank one.
Thus the restriction of the non-zero vector field $\eta|_{S(f)}$
near $p$ can be parameterized
by the parameter $t$ of $\gamma$.
We set
\begin{equation}\label{eq:phi}
\phi(t)=\det(\gamma'(t),\eta(t)).
\end{equation}
A non-degenerate singular point is said to be of the
{\it first kind\/} if $\phi(0)\ne0$, and
a non-degenerate singular point is said to be of the
$k$-{\it th kind\/} if $\phi(0)=0$, $\phi^{(j)}(0)=0$ $(j=1,\ldots,k-2)$,
and $\phi^{(k-1)}(0)\ne0$.
Then the following fact holds.
\begin{fact}[{\cite[Proposition 1.3]{krsuy},\cite[Corollary A.9]{is}}]
\label{fact:krsuy}
Let\/ $p$ be a non-degene\-rate singular point of
a front\/ $f$.
Then
\begin{itemize}
\item[\rm (i)] $f$ at\/ $p$ is a cuspidal edge if and only if\/
$p$ is a singular point of the first kind.
\item[\rm (ii)]  $f$ at\/ $p$ is a swallowtail if and only if\/
$p$ is a singular point of the second kind.
\item[\rm (iii)]  $f$ at\/ $p$ is a cuspidal butterfly if and only if\/
$p$ is a singular point of the third kind.
\end{itemize}
\end{fact}
On the other hand, cuspidal lips and beaks are examples which
are degenerate singularities.
The following fact holds.
\begin{fact}[{\cite[Theorem A.1]{ist}}]\label{fact:ist}
Let\/ $p$ be a corank one singular point of
a front\/ $f$.
Then
\begin{itemize}
\item[\rm (i)]  $f$ at\/ $p$ is a cuspidal lips if and only if\/
$\det \hess \lambda \, (p)>0$.
\item[\rm (ii)]  $f$ at\/ $p$ is a cuspidal beaks if and only if\/
$\det \hess \lambda \, (p)<0$ and\/ $\eta\eta\lambda(p)\ne0$,
\end{itemize}
where\/ $\eta$ is a null-vector field,
and\/ $\lambda$ is a singularity identifier.
\end{fact}

For a front $f$,
the mean curvature $H$
diverges near non-degenerate singular points.
Furthermore, the following fact is known.
\begin{fact}[{\cite[Propositions 3.8, 3.11, 4.2 and (4.6)]{MSUY}}]
Let\/ $p$ be a non-dege\-nerate singular point of
a front\/ $f$.
Then\/ $\hat H=\lambda H$ is a\/ $C^\infty$-function,
and\/ $\hat H(p)\ne0$, where\/ $\lambda$ is a singularity
identifier.
\end{fact}
For degenerate singularities with some conditions,
which include the cuspidal lips and beaks, we have the same
statement.

\begin{proposition}
Let\/ $p$ be a corank one singular point of
a front\/ $f$, satisfying\/
$d\lambda_p=0$,
where\/ $\lambda$ is a singularity identifier.
Then\/ $\hat H=\lambda H$ is a\/ $C^\infty$-function near\/ $p$,
and\/ $\hat H(p)\ne0$.
\end{proposition}
\begin{proof}
Since $p$ is a  corank one singular point of $f$,
we may assume that  
$p=(0,0)$ and
$f$ has the form
$$
f(u,v)=\big(u,f_2(u,v),f_3(u,v)\big),\quad
\big(d(f_2)_{(0,0)}=d(f_3)_{(0,0)}=0\big).
$$
Then we see that we can take $\partial_v$ as a null vector
field, and the unit normal vector
$\nu=(\nu_1,\nu_2,\nu_3)$
satisfies $\nu_1(p)=0$.
Thus, we may further assume that $\nu_2(p)=0$, $\nu_3(p)=1$
by a rotation of $\R^3$.
Since $\inner{f_v(u,v)}{\nu(u,v)}=
(f_2)_v(u,v)\nu_2(u,v)+(f_3)_v(u,v)\nu_3(u,v)=0$,
there exists a non-zero function $l(u,v)$ such that
$$
(f_3)_v(u,v)=l(u,v)(f_2)_v(u,v).
$$
We set $g(u,v)=(0,1,l(u,v))$ and $q(u,v)=(f_2)_v(u,v)$.
Then
the first fundamental matrix ${\rm I}$
and
the second fundamental matrix ${\rm I\!I}$
are
\begin{align*}
{\rm I}&=\pmt{
1+(f_2)_u^2+(f_3)_u^2&q\inner{f_u}{g}\\
q\inner{f_u}{g}&q^2(1+l^2)},
\\
{\rm I\!I}
&=\pmt{
-\inner{f_u}{\nu_u}&-q\inner{g}{\nu_u}\\
-q\inner{g}{\nu_u}
&
-q\inner{g}{\nu_v}}.
\end{align*}
Hence the mean curvature can be computed as
\begin{equation}\label{eq:meanlb}
H=
\dfrac
{(1+(f_2)_u^2+(f_3)_u^2)\inner{g}{\nu_v}
+q*}
{2q
\Big((1+(f_2)_u^2+(f_3)_u^2)(1+l^2)-\inner{f_u}{g}^2\Big)},
\end{equation}
where $*$ stands for a function which is not necessary
in the later calculations.
Since
$$
\det(f_u,f_v,\nu)
=
q\det
\pmt{f_u, g,\nu},
$$
and
$
\det\pmt{f_u, g,\nu}(p)
=\nu_3(p)\ne0$, we see that
$q$ is a singularity identifier.
Thus, the first assertion is shown.
By \eqref{eq:meanlb}, to show the second assertion,
showing $\inner{g}{\nu_v}(p)\ne0$ is enough.
Since
$\det(f_u,g,\nu)(p)\ne0$,
the vectors $f_u(p),g(p),\nu(p)$ form a basis
of $\R^3$,
and by
$\inner{f_u}{\nu_v}(p)=\inner{f_v}{\nu_u}(p)=0$ and
$\inner{\nu}{\nu_v}(p)=0$,
we see that
$\inner{g}{\nu_v}(p)\ne0$
is equivalent to
$\nu_v(p)\ne0$.
Since $f$ at $p$ is a front and $f_v(p)=0$,
we get $\nu_v(p)\ne0$.\qed
\end{proof}

We remark that a map-germ $f:(\R^2,p)\to(\R^3,0)$ 
whose set of singular points is nowhere dense
is a frontal
if and only if the Jacobi ideal of $f$ is 
principal \cite[Lemma 2.3]{ishifrontal}.

\subsection{Geometric invariants of fronts}

Here we review the known geometric invariants
of cuspidal edges and swallowtails.
Singular points of cuspidal edges are non-degenerate and
several geometric invariants are defined and studied.
Let $f:(\R^2,p)\to(\R^3,0)$ be a front
and $\nu$ the unit normal vector.
\begin{definition}\label{def:coord01}
Let $p$ be a non-degenerate singular point of $f$.
A positively oriented coordinate system $(u,v)$ centered at $p$
is said to be
{\it $u$-singular\/}
if $S(f)=\{(u,v)\,|\,v=0\}$ holds.
Let $p$ be a singular point of the first kind.
A $u$-singular coordinate system $(u,v)$ is said to be
{\it strongly adapted\/} if
the null vector on $S(f)$ is $\partial_v$.
\end{definition}
Existence of $u$-singular and strongly adapted coordinate systems
are easily shown.
We assume that $(u,v)$ is a strongly adapted coordinate system.
We set
\begin{align*}
\kappa_s(u)&=
\sgn\big(\det(f_{u},f_{vv},\nu)\big)
\dfrac{\det(f_u,f_{uu},\nu)}
{|f_{u}|^3}\Big|_{(u,0)},\\
\kappa_\nu(u)&=\dfrac{\inner{f_{uu}}{\nu}}
{|f_{u}|^2}\Big|_{(u,0)},\\
\kappa_t(u)&=
\dfrac{\det(f_{u},\,f_{vv},\,f_{vvu})}
{|f_{u}\times f_{vv}|^2}
-
\dfrac{\det(f_{u},\,f_{vv},\,f_{uu})
\inner{f_{u}}{f_{vv}}}
{|f_{u}|^2|f_{u}\times f_{vv}|^2}\Big|_{(u,0)},\\
\kappa_c(u)&=\dfrac{|f_{u}|^{3/2}\det(f_{u},f_{vv}, f_{vvv})}
{|f_u \times  f_{vv}|^{5/2}}\Big|_{(u,0)}.
\end{align*}
All these functions are differential geometric invariants,
and
$\kappa_s$ is called the {\it singular curvature},
$\kappa_\nu$ is called the {\it limiting normal curvature},
$\kappa_t$ is called the {\it cuspidal torsion\/}
({\it cusp-directional torsion\/}), and
$\kappa_c$ is called the {\it cuspidal curvature}.
See \cite{MSUY, MS, front} for detail.
Next we assume that $p$ is a singular point of the $k$-th kind, $k\geq2$.
We take a $u$-singular coordinate system $(u,v)$.
Since $p$ is a singular point of the $k$-th  kind, $k\geq2$,
$\partial_u$ is a null vector at $p$.
Taking $k = 2$, we set
\begin{align}
\mu_c (p) &=
\dfrac{-|f_v|^3\inner{f_{uv}}{\nu_u}}
{|f_{uv}\times f_{v}|^2}\Big|_{p},\label{eq:muc}\\
\tau_s (p) &=
\dfrac{|
\det(f_{uu},f_{uuu},\nu)|}{|f_{uu}|^{5/2}}
\Big|_{p}\, .\label{eq:taus}
\end{align}
The constant $\mu_c$ is called the {\it normalized cuspidal
curvature\/}, and it relates the boundedness of
the mean curvature.
The constant $\tau_s$ is called the {\it limiting singular
curvature\/}, and it measures the wideness of
the cusp of the swallowtail.
See \cite[Section 4]{MSUY} for detail.

\subsection{Kossowski metric}\label{sec:kos}
Let $f:(\R^2,p)\to(\R^3,0)$ be a front-germ.
Then the induced metric $f^*(\inner{~}{~})$ on $\R^2$
is positive semi-definite.
An intrinsic formulation of this metric is
called the Kossowski metric and studied
\cite{kos,hhnsuy}.
Let $g$ be
a positive semi-definite metric-germ at $p\in\R^2$.
A point $q$ is called a {\it singular point\/}
of $g$ if 
the metric $g$ is not positive definite at $q$.
The set of singular points of $g$ is denoted by $S(g)$.
For a singular point $q\in S(g)$,
the subspace
$$
N_q=\{v\in T_q\R^2\,|\, g(v,w)=0 \text{ for all }
w\in T_q\R^2\}
$$
is called the {\it null space\/} at $q$, and
a non-zero vector in $N_q$ is called {\it null vector\/}
at $q$.
A singular point $q$ of $g$ is said to be {\it corank one\/}
if the dimension of $N_q$ is one.
If $q$ is of corank one, then there exists a non-zero vector
field $\eta$ which is a generator of $N_q$.
We call $\eta$ a {\it null vector field}.

\begin{definition}[{\cite[Section 2]{hhnsuy}}]\label{def:metric}
(1) The metric $g$ is {\it admissible\/}
if there exists a coordinate system
$(u,v)$ on a neighborhood of $0$ in $\R^2$ 
such that
$$
F=G=0,\quad
E_v=2F_u,\quad
G_u=G_v=0
$$
hold on $S(g)$,
where 
\begin{equation}\label{eq:efgofg}
E=g(\partial_u,\partial_u),\quad
F=g(\partial_u,\partial_v),\quad
G=g(\partial_v,\partial_v).
\end{equation}
(2) An admissible metric is called a
{\it frontal metric\/}
if there exists a coordinate system
$(u,v)$ on a neighborhood of $0$ in $\R^2$ 
and
there exists a function $\lambda$ such that
\begin{equation}\label{eq:koslambda}
EG-F^2=\lambda^2
\end{equation}
holds on $(\R^2,0)$.
(3) A singular point $p$ of an admissible metric $g$
is called {\it non-degenerate\/}
if
$$
d\lambda_p\ne0
$$
holds,
where $\lambda$ is the function as in \eqref{eq:koslambda}.
(4) A frontal metric is called a {\it Kossowski metric\/}
if all singular points are non-degenerate.
\end{definition}

We remark that the definition
(1) of Definition \ref{def:metric}
is not the same as the original.
The original definition is coordinate free.
See \cite[page 103]{kos} and \cite[Definition 2.3]{hhnsuy}.

\begin{proposition}[{\cite[Proposition 2.10]{hhnsuy}}]
Let\/ $f:(\R^2,p)\to(\R^3,0)$ be a front.
Then the induced metric\/ $f^*(\inner{~}{~})$ is a frontal metric.
If\/ $p$ is a non-degenerate singular point,
then it is a Kossowski metric.
\end{proposition}

Let $p$ be a non-degenerate singular point.
Like as the case of fronts, we give the following definition.
Since $p$ is non-degenerate,
there exists a regular curve-germ
$\gamma:(\R,0)\to (\R^2,p)$ such that
the image of $\gamma$ coincides with the set of singular points
$S(g)$,
and there exists a non-zero vector field $\eta$
near $p$ such that $\eta(q)$ is a null vector for $q\in S(g)$.
We set
\begin{equation}\label{eq:psi}
\psi(t)=\det(\gamma'(t),\eta(t)).
\end{equation}
A non-degenerate singular point is called an $A_2$-{\it point\/}
if $\psi(0)\ne0$, and
a non-degenerate singular point is called an $A_k$-{\it point\/}
if $\psi(0)=0$, $\psi^{(j)}(0)=0$ $(j=1,\ldots,k-3)$,
and $\psi^{(k-2)}(0)\ne0$.
Let $p$ be a corank one but not a non-degenerate singular point 
(i.e., corank one and $d\lambda_p=0$)
of the frontal metric $g$.
Then $p$ is said to be {\it Morse type}
if $\det\hess \lambda(p)\ne0$
and $\eta\eta\lambda(p)\ne0$,
where $\lambda$ is the function defined  by \eqref{eq:koslambda} in
Definition \ref{def:metric}.
We remark that the condition of Morse type
has an additional condition
$\eta\eta\lambda(p)\ne0$ not only $\lambda$ has
a Morse type critical point at $p$.

\section{Kenmotsu-type formula for singular surfaces}
\subsection{Kenmotsu-type formula}\label{sec:kenmotsu}

With the terminology in Section \ref{sec:kos},
we give a Kenmotsu-type formula following
Kokubu \cite{koku}.
Let $U\subset \R^2$ be a simply-connected open set, and
$g$ be a positive semi-definite metric on $U$.
Set $E,F$ and $G$ as in 
\eqref{eq:efgofg}
for a coordinate system $(u,v)$.
We assume that $g$ is a frontal metric,
and take a function $\lambda$ defined by \eqref{eq:koslambda}.
Although $\lambda$ has two choices,
in what follows we fix one of them.
Let $H:U\setminus S(g)\to\R\setminus\{0\}$
be a $C^\infty$-function which satisfies that
\begin{equation}\label{eq:hhat}
\dfrac{1}{-2H(u,v)}
=
\dfrac{\lambda(u,v)}{\hat{H}(u,v)}
\end{equation}
for a non-zero function $\hat{H}$.
Let $\nu:U\to S^2$
be a unit vector valued function, where $S^2=\{X\in \R^3\,|\,|X|=1\}$.
We assume $g$ is of corank one at $p$.
Let $\eta$ be a null vector of $g$ at $p$.
The pair $(g,\nu)$ is called a {\it front pair\/} if
\begin{equation}\label{eq:frontpair}
\eta \nu (p)\ne0.
\end{equation}
It is easy to see that
if $g$ is a induced metric of a front $f$ and $\nu$
is its unit normal vector, then
$(g,\nu)$ is a front pair.
The following theorem gives a kind of recipe to 
obtain a surface with prescribed mean curvature and Kossowski metric.
Since Kossowski metric contains information of 
the class of singularity, one can obtain a desired surface with singular points
which has prescribed mean curvature and given types of singular points.
\begin{theorem}\label{thm:construction}
Let\/ $g$ be a Kossowski metric\/ $($\/respectively, frontal metric\/$)$ 
on a simply-connected open set\/ $U(\subset\R^2)$, 
and let\/ $p\in U$ be an\/ $A_{k+1}$-point\/ $($\/respectively, 
a Morse type singular point\/$)$ of $g$. 
Let\/ $\lambda$ a function defined by\/ \eqref{eq:koslambda},
and let\/ $H$ be a function satisfying\/ \eqref{eq:hhat} 
for some non-zero function\/ $\hat{H}$. 
Let\/ $\nu$ be a unit vector valued map so that\/ $(g,\nu)$ 
is a front pair and\/ $d\nu\neq0$ on\/ $U$.
Assume that\/ $g,H,\nu$ satisfy the integrability condition
\begin{equation}\label{eq:integ}
\dfrac{\partial}{\partial v}
\Big(\dfrac{1}{\hat{H}}
\big(
\lambda \nu_u+F\nu\times\nu_u-E\nu\times\nu_v
\big)\Big)
=
\dfrac{\partial}{\partial u}
\Big(\dfrac{1}{\hat{H}}\big(
\lambda \nu_v+G\nu\times\nu_u-F\nu\times\nu_v
\big)\Big),
\end{equation}
and set 
\begin{equation}\label{eq:kokuken}
f(u,v)=
\int
\dfrac{1}{\hat{H}}\Big(
\big(
\lambda \nu_u+F\nu\times\nu_u-E\nu\times\nu_v
\big)\,du
+
\big(
\lambda \nu_v+G\nu\times\nu_u-F\nu\times\nu_v
\big)\,dv\Big).
\end{equation}
Then\/ $f:U\to\R^3$ is a front, and\/ $p$ is a\/ $k$-th kind 
singular point\/ $($\/respectively, cuspidal lips/beaks\/$)$.
Moreover, the mean curvature of\/ $f$
on the set of regular points coincides with\/ $H$,
the Gauss map is\/ $\nu$, it holds that\/
$S(f)=S(g)$, 
and the induced metric\/ $f^*(\inner{~}{~})$ is
proportional to\/ $g$.
\end{theorem}

\begin{proof}
Since
\begin{equation}\label{eq:fufv}
f_u=
\dfrac{1}{\hat{H}}\Big(
\lambda \nu_u+F\nu\times\nu_u-E\nu\times\nu_v\Big),\quad
f_v=\dfrac{1}{\hat{H}}
\Big(\lambda  \nu_v+G\nu\times\nu_u-F\nu\times\nu_v\Big),
\end{equation}
$\nu$ gives the unit normal vector of $f$.
Setting
\begin{equation}
\label{eq:pqrd}
\begin{array}{rlrlrlrl}
P=& \inner{\nu_u}{\nu_u},\ &
Q=& \inner{\nu_u}{\nu_v},\ &
R=& \inner{\nu_v}{\nu_v},\ &
D=& \det(\nu_u,\,\nu_v,\,\nu),\\
E_f=& \inner{f_u}{f_u},\ &
F_f=& \inner{f_u}{f_v},\ &
G_f=& \inner{f_v}{f_v},\\
L_f=& \inner{f_{uu}}{\nu},\ &
M_f=& \inner{f_{uv}}{\nu},\ &
N_f=& \inner{f_{vv}}{\nu},
\end{array}
\end{equation}
and
$$
X=2\lambda D+GP-2FQ+ER,
$$
we have
\begin{align*}
E_f
&=
\dfrac{X}{\hat{H}^2}E, &
F_f
&=
\dfrac{X}{\hat{H}^2}F, &
G_f
&=
\dfrac{X}{\hat{H}^2}G, \\
L_f&=-\dfrac{1}{\hat{H}}\left(\lambda P+ED\right), &
M_f&=-\dfrac{1}{\hat{H}}\left(\lambda Q+FD\right), &
N_f&=-\dfrac{1}{\hat{H}}\left(\lambda R+GD\right).
\end{align*}
Thus, we get
$$
H_f
=
\dfrac{E_fN_f-2F_fM_f+G_fL_f}{2(E_fG_f-F_f^2)}
=
-\dfrac{\hat{H}}{2\lambda}=H
$$
on the set of non-singular points by using equations $\lambda^2=EG-F^2$ and $GP-2FQ+ER=X-2\lambda D$, and
$$
\det(f_u,f_v,\nu)=
\dfrac{\lambda}{\hat{H}^2}X.
$$

We now consider the function $X$. 
Since $g$ is an frontal metric on $U$, 
it follows that $\eta=\partial_v$, namely, $E\neq0$ and $\lambda=F=G=0$ on $S(g)$ (cf. Definition \ref{def:metric}).
Since $(g,\nu)$ is a front pair, $\eta\nu=\nu_v\neq0$ on $S(g)$, and hence $R\neq0$ on $S(g)$. 
Therefore $X\neq0$ on $S(g)$. 
Thus we see that $f$ is a front. 
Moreover, since $\eta\nu\ne0$, in particular, 
$d\nu\neq0$ on $U$, the function $X$ does not vanish on $U$. 
Thus $S(f)=S(g)$ holds on $U$. 
Further, it follows that $\eta f=f_v=0$ on $S(f)$. 
Thus we can also regard $\eta$ as a null vector field of $f$.

We consider types of singularities of $f$. 
First we assume that $g$ is a Kossowski metric on $U$ and $p$ is an $A_{k+1}$-point of $g$. 
Then $d\lambda_p\neq0$. 
Thus there exists a regular curve $\gamma(t)$ such that $\gamma$ parametrizes $S(g)=S(f)$ on $U$. 
In this situation, the functions $\phi(t)$ and $\psi(t)$ as in \eqref{eq:phi} and \eqref{eq:psi}, respectively, are the same. 
By definitions of $k$-th kind and $A_{k+1}$-point, 
it holds that $f$ has a $k$-th kind singularity at $p$. 

We next assume that $p$ is a Morse type singular point of $g$. 
Then $d\lambda_p=0$, $\det\hess \lambda(p)\neq0$ and $\eta\eta\lambda(p)\neq0$ hold. 
On the other hand, by the above calculation, it holds that $\det(f_u,f_v,\nu)=\lambda\cdot(\text{non-zero function})$. 
Thus by Fact \ref{fact:ist}, $f$ has a cuspidal lips or a cuspidal beaks at $p$.\qed
\end{proof}

Since the function $\lambda$ defined by \eqref{eq:koslambda}
has the plus-minus ambiguity, 
by \eqref{eq:kokuken} the surface
$f$ also has two possibilities, say $f_+$ and $f_-$,
for a semi-definite metric $g$ and single $H$, $\nu$.
In this case, considering $(u,-v)$ instead of $(u,v)$,
then we see that $-\lambda$ satisfies the integrability
condition, and
the resulting surface $f_-(u,v)$ coincides
with $f_+(u,-v)$.

\begin{example}
Let us set $J=(-1/2,1/2)$ and
\begin{align*}
&H(v)=-\dfrac{1}{2\sinh v}(-3+\cosh 2v),\quad
E(v)=1/\cosh^2v,\quad
F(v)=0,\\
&G(v)=\sinh^2v/\cosh^2v,\quad
\lambda(v)=\sinh v/\cosh^2v
\end{align*}
for $v\in J$.
Then we see $\lambda(v)^2=E(v)G(v)-F(v)^2$, and
$\hat H(v)=-2H(v)\lambda(v)=(-3+\cosh 2v)/\cosh^2v$
is a non-zero $C^\infty$-function on $J$.
Regarding these functions as functions on $\R\times J$,
we can see that the null vector field of $g=Edu^2+2Fdudv+Gdv^2$
is $\partial_v$, and each point on $\{(u,0)\}$ is an $A_2$-singularity.
Let us set $\nu(u,v)=(\cos u \sinh v,\sin u \sinh v,1)/\cosh v$.
Then $\eta\nu\ne0$ holds.
Moreover, these functions satisfy the condition \eqref{eq:kokuken}.
By Theorem \ref{thm:construction}, there exists a
surface $f:\R\times J\to \R^3$ whose mean curvature
is $H$ and $S(f)=\{(u,0)\}$.
Since $(u,0)$ is an $A_2$-type singularity of $g$,
it holds that $f$ at $(u,0)$ is a cuspidal edge.
\end{example}

Here, we give an application of Theorem \ref{thm:construction}.
Let $I\subset \R$ be an interval, and $S\subset I$ be a
discrete subset.
Taking two functions $H:I\setminus S\to \R\setminus\{0\}$
and
$l:I\to \R$ satisfying $l^{-1}(0)=S$ and
$lH$ is a non-zero $C^\infty$ function on $I$.
We set $\hat H=-2lH$, and $\nu(u,v)=(\sin\theta(u),-\cos\theta(u),0)$,
where $\theta(u)$ is a primitive function of $\hat H$,
and $(u,v)\in I\times\R$.
We also set 
$$
E(u,v)=l(u)^2,\quad F(u,v)=0,\quad G(u,v)=1\quad
((u,v)\in I\times\R).
$$
Then these functions satisfy the condition \eqref{eq:kokuken}.
By Theorem \ref{thm:construction}, there exists a
surface $f:I\times \R\to \R^3$ whose mean curvature
is $H$ and 
$f=(f_1(u),f_2(u),v)$, 
where $f_1$ (respectively, $f_2$) 
is a primitive function of $l(u)\cos\theta(u)$ (respectively, $l(u)\sin\theta(u)$).
Hence for any function pair $H:I\setminus S\to \R\setminus\{0\}$,
and
$l:I\to \R$ satisfying $I^{-1}(0)=S$ and
$lH$ is a non-zero $C^\infty$ function on $I$,
there exists a surface which is symmetric with respect to a translation,
whose mean curvature is $H$ and the singular set is $S\times \R$.
Setting $H=-1/\sin u$, and $l=(\sin u)/2$ (respectively, $l=(11/10)\sin u$), 
then
we obtain a surface
$f(u,v)=(-(\cos^2u)/4,(u-\cos u\sin u)/4,v)$
(respectively, 
$f(u,v)=(
11(8 \cos u_1-3 \cos u_2)/192 ,11(8 \sin u_1-3 \sin u_2)/192,v)$,
where $u_1=6u/5, u_2=16 u/5$)
which is illustrated in Figure \ref{fig:sin}, left (respectively, right).
\begin{figure}[ht]
\centering
\includegraphics[width=.5\linewidth]{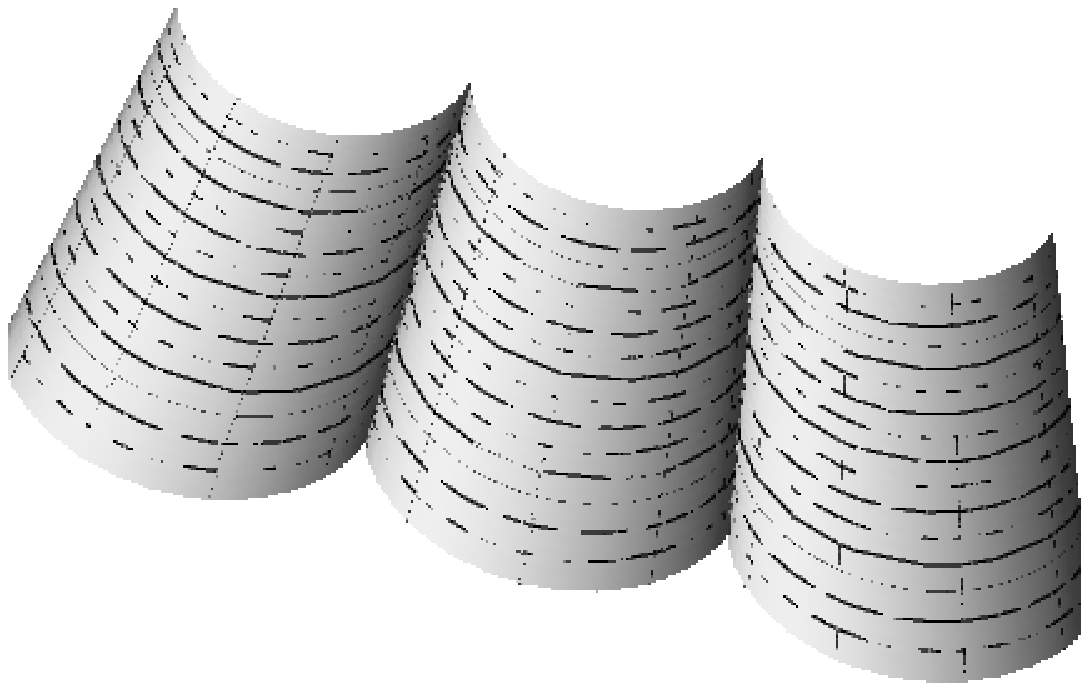}
\hspace{10mm}
\includegraphics[width=.3\linewidth]{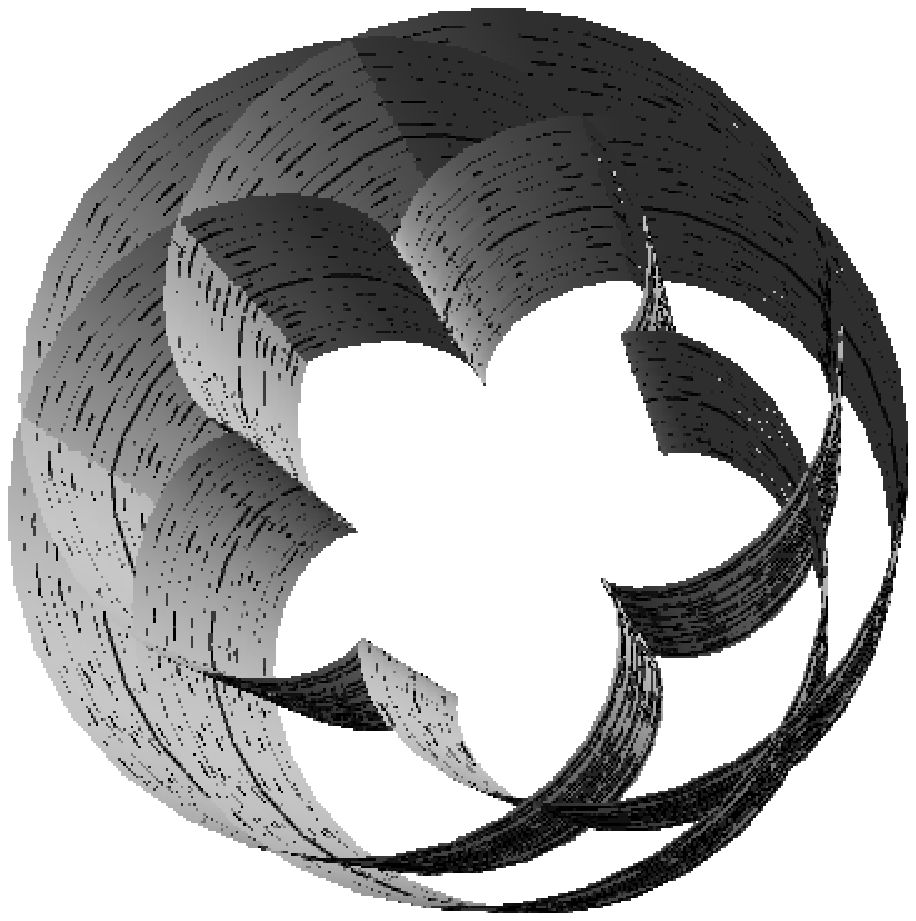}
\caption{Surfaces whose mean curvatures are $-1/\sin u$.} 
\label{fig:sin}
\end{figure}

\vspace*{-0.8cm}
\subsection{Invariants of cuspidal edge}
In this section, we calculate geometric invariants of $f$ for the
case that $f$ is a cuspidal edge. We take $g,H,\nu$ satisfying the
assumption of Theorem \ref{thm:construction}, and fix a function
$\lambda$ defined by \eqref{eq:koslambda}. Let $f$ be the front as
in \eqref{eq:kokuken}. We assume further that $p$ is a singular
point of the first kind of $f$, namely, $f$ at $p$
is a cuspidal edge. Similar to Definition \ref{def:coord01}, we
define coordinate systems.
\begin{definition}\label{def:coord02}
Let $p$ be a non-degenerate singular point of a metric $g$.
A positively oriented coordinate system $(u,v)$ is said to be
{\it $u$-singular\/}
if the set of singular points satisfies
$S(g)=\{(u,v)\,|\,v=0\}$.
Let $p$ be an $A_2$-type singular point.
A $u$-singular coordinate system $(u,v)$ centered at $p$
is said to be
{\it strongly adapted\/} if
the null vector on $S(g)$ is $\partial_v$.
\end{definition}

See \cite[Section 2]{hhnsuy} for existence of a strongly
adapted coordinate system. Then we have the following proposition.

\begin{proposition}\label{prop:invariants}
Let\/ $p$ be an\/ $A_2$-type singular point, and\/
$(u,v)$ a strongly adapted coordinate system.
Then the invariants\/
$\kappa_\nu$, $\kappa_s$, $\kappa_c$ and\/ $\kappa_t$
can be calculated as
\begin{align}
\kappa_\nu(u)
=&-\frac{D\hat{H}}{ER}\bigg|_{(u,0)},\label{eq:propknu}\\
\kappa_s(u)=&
\sgn\big(\det(f_u,f_{vv},\nu)\hat{H}\big)
\frac{\det(\nu,\nu_v,\nu_{uv})\hat{H}}
{ER^{3/2}}\bigg|_{(u,0)}
,\label{eq:propks}\\
\kappa_c(u)=& -\sgn\big(\hat{H}\big)
\dfrac{2|\hat{H}|^{1/2}R^{1/4}}{|\lambda_v|^{1/2}}
\bigg|_{(u,0)},\label{eq:propkc}\\
\kappa_t(u) =&
-\frac{\hat{H}Q}
{ER}\bigg|_{(u,0)}.\label{eq:propkt}
\end{align}
Here, $E,F,G$ are defined in\/ \eqref{eq:efgofg},
and\/ $P,Q,R,D$ are defined in\/ \eqref{eq:pqrd}.
\end{proposition}
\begin{proof}
Since $(u,v)$ is a strongly adapted coordinate system, 
it holds that $f_v=0$ on the $u$-axis.
Thus $F=G=0$ and $F_u=G_u=G_v=0$ on the $u$-axis, and $\lambda$ is a non-zero
functional multiplication of $v$. So, $\lambda = \lambda_u = 0$
on the $u$-axis. We set $1/\hat{H}=a$. Firstly we consider
$\kappa_\nu$ and $\kappa_s$ of $f$. It holds that $f_u=-aE
\nu\times\nu_v$ on the $u$-axis. So,
\begin{equation}
|f_u| = |a|ER^{1/2}
 \label{eq:fu}
 \end{equation}
 on the $u$-axis.
 Since
\begin{align*}
f_{uu}=&(\lambda_ua+\lambda a_u)\nu_u
+\lambda a\nu_{uu}+(aF_u+a_uF)\nu\times\nu_u\nonumber\\
&\hspace{30mm}
+aF(\nu\times\nu_u)_u-(aE_u+a_uE)\nu\times\nu_v-aE(\nu\times\nu_v)_u,
\end{align*}
and $F_u=0$ on the $u$-axis, we have 
\begin{equation}
f_{uu}=-(aE_u+a_uE)\nu\times\nu_v-aE(\nu_u\times\nu_v+\nu\times\nu_{uv})
\label{eq:fuu}
 \end{equation}
on the $u$-axis. 
Thus by \eqref{eq:fu} and \eqref{eq:fuu}, we
have $\kappa_\nu(u) =-D/(aER) (u,0)$, and it proves \eqref{eq:propknu}.
Furthermore, since
$$f_u\times\nu=-aE\nu_v$$
on the $u$-axis,
it follows that
$$\det(f_u,f_{uu},\nu)=-\inner{f_{uu}}{f_u\times\nu}
=a^2E^2\det(\nu,\nu_v,\nu_{uv})$$
on the $u$-axis.
Therefore,
\begin{equation*}
\kappa_s(u)=
\sgn(\det(f_u,f_{vv},\nu))
\frac{\det(\nu,\nu_v,\nu_{uv})}{|a|ER^{3/2}} (u,0),
\end{equation*}
and it proves \eqref{eq:propks}.

Next we consider $\kappa_c$.
By direct calculations, we have:
\begin{align*}
f_{vv}&=(\lambda_va+\lambda a_v)\nu_v
        +\lambda a\nu_{vv}+(aG_v+a_vG)\nu\times\nu_u\\
&\hspace{5mm}
+aG(\nu\times\nu_u)_v-(aF_v+a_vF)\nu\times\nu_v-aF(\nu\times\nu_v)_v,\\
f_{vvv}&=(\lambda_{vv}a+2\lambda_va_v+\lambda a_{vv})\nu_v
+2(\lambda a_v+\lambda_v a)\nu_{vv}+\lambda a\nu_{vvv}\\
&\hspace{5mm}
+(aG_{vv}+2a_vG_v+a_{vv}G)\nu\times\nu_u
+2(aG_v+a_vG)(\nu\times\nu_u)_v\\
&\hspace{5mm}
+aG(\nu\times\nu_u)_{vv}
-(aF_{vv}+2a_vF_v+a_{vv}F)\nu\times\nu_v
\\
&\hspace{5mm}
-2(aF_v+a_vF)(\nu\times\nu_v)_v-aF(\nu\times\nu_v)_{vv}.
\end{align*}
So,
\begin{align}
f_{vv}&=\lambda_va\nu_v-aF_v\nu\times\nu_v,
\label{fvv}\\
f_{vvv}&=(\lambda_{vv}a+2\lambda_va_v)\nu_v
         +2\lambda_va\nu_{vv}+aG_{vv}\nu\times\nu_u\nonumber\\
&\hspace{20mm}
-(aF_{vv}+2a_vF_v)\nu\times\nu_v-2aF_v(\nu\times\nu_v)_v
\label{fvvv}
\end{align}
on the $u$-axis.
Therefore, it holds that
\begin{equation}\label{eq:fufvv}
f_u\times f_{vv}=
-a^2E\lambda_v(\nu\times\nu_v)\times\nu_v=a^2E\lambda_vR\nu
\end{equation}
on the $u$-axis. Noticing that $R= - \inner{\nu}{\nu_{vv}}$,
we have
$$\det(f_u,f_{vv},f_{vvv})=-2a^3E\lambda_v^2R^2.$$

On the other hand, by \eqref{eq:fu} and \eqref{eq:fufvv}, it
follows that
$$
|f_u|^{3/2}=(|a|ER^{1/2})^{3/2},\quad
|f_u\times f_{vv}|^{5/2}=a^4E^{5/2}\lambda_v^2R^{5/2}
|a||\lambda_v|^{1/2},
$$ 
on the $u$-axis. Hence we obtain
\begin{align*}
\kappa_c(u)&=
\frac{-2(|a|ER^{1/2})^{3/2}a^3E\lambda_v^2R^2}
{a^4E^{5/2}\lambda_v^2R^{5/2}|a||\lambda_v|^{1/2}} (u,0)\\
& =
\frac{-2|a|^{3/2}R^{1/4}} {a|a||\lambda_v|^{1/2}} (u,0) =
\frac{-2|a|^{1/2}R^{1/4}} {a |\lambda_v|^{1/2}} (u,0)
\end{align*}
which proves \eqref{eq:propkc}.
Next we consider $\kappa_t$.
Since
\begin{align}
f_{uv}=& a \lambda_v  \nu_u + aF_v \nu \times \nu_u
- (a E)_v \nu \times \nu_v + \lambda(a_v \nu_u + a\nu_{uv}) \nonumber \\
&\hspace{10mm}
+ F(a_v \nu \times \nu_u - a \nu_u \times \nu_v
+ a \nu \times \nu_{uv}) - aE\nu \times \nu_{vv},\nonumber
\end{align}
using \eqref{eq:fuu},
\eqref{eq:fufvv},  and noticing that $Q=-\inner{\nu}{\nu_{uv}}$, 
we obtain
\begin{align}
\inner{f_u}{f_{vv}} =& \, a^2 E F_vR \label{eq:fudotfvv}\\
\det(f_u,f_{vv},f_{uvv})=& \inner{a^2 E \lambda_v R \nu}{2a\lambda_v
\nu_{uv}
- 2aF_v \nu_u \times \nu_v - a E \nu_v \times \nu_{vv}}\nonumber \\
=& -a^3\lambda_v ER( 2 \lambda_v Q + 2F_vD + E\det(\nu, \nu_v ,
\nu_{vv})),
\label{eq:fufvvfuvv}\\
\det(f_u,f_{vv},f_{uu}) =& \inner{a^2 E \lambda_v R \nu}{-aE
\nu_u\times \nu_v} = -a^3\lambda_v E^2 R D\label{eq:fufvvfuu}
\end{align}
on the $u$-axis. By \eqref{eq:fu},\eqref{eq:fufvv},
\eqref{eq:fudotfvv}, \eqref{eq:fufvvfuvv} and \eqref{eq:fufvvfuu},
we have
\begin{align}
\kappa_t(u) =& \, \frac{-a^3\lambda_v ER ( 2 \lambda_v Q + 2F_vD +
E\det(\nu, \nu_v , \nu_{vv}))} {a^4\lambda_v^2 E^2 R^2} (u,0)\nonumber\\
& + \frac{(a^3\lambda_v E^2 R D)(a^2EF_vR)}
{|a^2E^2R|\,|a^4\lambda_v^2 E^2 R^2|} (u,0) \nonumber \\
=& \,
\frac{-2\lambda_v Q - F_vD 
- E\det(\nu, \nu_v ,\nu_{vv})} {a \lambda_v  E R}(u,0).
\label{eq:ktpre}
\end{align}
On the other hand, by the integrability condition,
we have
\begin{equation}\label{eq:integsing}
-a_vE\nu\times\nu_v
+a(\lambda_v\nu_u+F_v\nu\times\nu_u-E_v\nu\times\nu_v
-E\nu\times\nu_{vv})=0
\end{equation}
on the $u$-axis.
Taking the inner product with $\nu_v$, we have
$$\lambda_vQ+F_vD=-E\det(\nu,\nu_v,\nu_{vv}).$$
This equation together with \eqref{eq:ktpre}
proves the assertion.
\qed
\end{proof}

We have the following corollary.
\begin{corollary}
Under the same assumption in
Proposition\/ {\rm \ref{prop:invariants}},
the curve\/ $f(u,0)$
is
a part of a straight line if and only if
$$
\det(\nu_u,\nu_v,\nu)(u,0)=\det(\nu,\nu_v,\nu_{uv})(u,0)=0.
$$
The curve\/ $f(u,0)$ is a plane curve
if and only if
$$\dfrac{ER^2}{\hat{H}\det(\nu,\nu_v,\nu_{uv})^2+D^2R}V
+\hat{H}
\frac{-\lambda_v Q +2E\det(\nu, \nu_v ,\nu_{vv})}
{\lambda_v  E R}\Bigg|_{(u,0)}=0,
$$
where
$$
V=
\dfrac{\det(\nu,\nu_v,\nu_{uv})}{R^{1/2}}
\bigg(-\dfrac{D\hat{H}}{ER}\bigg)_u
+
\dfrac{D\hat{H}}{ER}
\bigg(\dfrac{\det(\nu,\nu_v,\nu_{uv})\hat{H}}
{ER^{3/2}}\bigg)_u.
$$
\end{corollary}
\begin{proof}
We set $\hat\gamma=f(\gamma)$.
Since the curvature $\kappa$ of $\hat{\gamma}$ as a space curve
satisfies $\kappa^2=\kappa_s^2+\kappa_\nu^2$ (\cite[Theorem 4.4]{MS}),
$\kappa=0$ is equivalent to $\kappa_s=\kappa_\nu=0$.
Since $\hat{H}\ne0$,
we have the first assertion by Proposition \ref{prop:invariants}.

We show the second assertion.
Since the torsion $\tau$ of $\hat{\gamma}$
satisfies
$$
\tau=
\dfrac{\kappa_s\kappa_\nu'-\kappa_s'\kappa_\nu}
{\kappa_s^2+\kappa_\nu^2}+\kappa_t
$$
(see \cite[Theorem 4.4]{MS}, \cite[(5.11)]{iste}),
we have the second assertion by Proposition \ref{prop:invariants}.\qed
\end{proof}

\subsection{Bounded of Gaussian curvature case}

In this section we study the case $f$ has
the bounded Gaussian curvature.
According to \cite[Corollary 3.12]{MSUY}
(see also \cite[Theorem 3.1]{front}),
it is equivalent to $\kappa_\nu=0$ on the set of singular points.
We assume that $(U,u,v)$ is an adapted coordinate system.
By Proposition \ref{prop:invariants},
$\kappa_\nu=0$ if and only if $D=0$. 
Since $f$ is a front, it is equivalent to
that
$\nu_u$ is parallel to $\nu_v$, say $\nu_u=\alpha\nu_v$ on the
$u$-axis.
If $D_v\ne0$ and $\alpha\ne0$, then $\nu$ is a fold.
Here, we shall see the case $\alpha=0$.
This implies that $\nu_u=0$, and by the integrability condition,
$\det(\nu,\nu_v,\nu_{vv})=0$ at the origin.
Then by Proposition \ref{prop:invariants},
$\kappa_t=0$ at the origin.
Furthermore, $\kappa_s=0$ if and only if
$D_v=\eta D=0$ at the origin.

\subsection{Invariants of swallowtail}

In this section, we calculate geometric invariants
of $f$ for the case that $f$ is a singular point of the
$k$-th kind ($k\geq2$).
We take $g,H,\nu$ satisfying the assumption of Theorem \ref{thm:construction},
and fix a function $\lambda$ defined by \eqref{eq:koslambda}.
Let $f$ be the front as in \eqref{eq:kokuken}.
We assume further that
$p$ is an $A_{k+1}$-type singular point of $g$ $(k\geq2)$.
We take a $u$-singular coordinate system centered at $p$.
Then we have the following proposition.
\begin{proposition}\label{prop:invariantssecond}
The invariants\/
$\mu_c$ and\/  $\tau_s$ $(k=2)$
can be calculated as
\begin{align}
\mu_c (p) =&
-\sgn\big(\hat{H}(p)\big) 
\dfrac{G(p)P(p)^{1/2}}
{\lambda_v(p)\hat{H}(p)^2},\label{eq:fmuc}\\
\tau_s (p) =&
\, \dfrac{|\hat{H}(p)|^{1/2}|2F_u(p)\det(\nu,\nu_u,\nu_{uu})(p)-E_{uu}(p)D(p)|}
{|F_u(p)|^{3/2}P(p)^{5/4}}
,\label{eq:ftaus}
\end{align}
where\/ $P$ and\/ $D$ are functions as in\/ \eqref{eq:pqrd}.
\end{proposition}

\begin{proof}
Since $(u,v)$ is a $u$-singular coordinate system,
we see $E=F=E_u=E_v=\lambda_u=0$
at $p$.
Since
\begin{align*}
f_{uv}=
&-\dfrac{\hat{H}_v}{\hat{H}}\Big(\lambda\nu_u+F\nu\times\nu_u-E\nu\times\nu_v\Big)\\
&\hspace{10mm}+\dfrac{1}{\hat{H}}\Big(\lambda_v\nu_u+\lambda\nu_{uv}+F_v\nu\times\nu_{u}
+F(\nu_v\times\nu_u+\nu\times\nu_{uv})\\
&\hspace{20mm}-E_v\nu\times\nu_v-E\nu\times\nu_{vv}\Big),
\end{align*}
thus we have 
$$f_{uv}(p)=\dfrac{1}{\hat{H}(p)}(\lambda_v(p)\nu_u(p)+F_v(p)(\nu\times\nu_u)(p)).$$
Then by \eqref{eq:fufv} and \eqref{eq:muc},
we have \eqref{eq:fmuc}.
Next, by direct calculations, we have
\begin{align*}
f_{uu}(p)&=\dfrac{F_u(p)}{\hat{H}(p)}(\nu\times\nu_u)(p),\\
f_{uuu}(p)&=-\dfrac{2\hat{H}_u(p)F_u(p)}{\hat{H}(p)^2}(\nu\times\nu_u)(p)
+\dfrac{1}{\hat{H}(p)}\Big(F_{uu}(p)(\nu\times\nu_u)(p)\\
&\hspace{30mm}
+2F_u(p)(\nu\times\nu_{uu})(p)-E_{uu}(p)(\nu\times\nu_v)(p)\Big).
\end{align*}
Then by \eqref{eq:taus}, we have \eqref{eq:ftaus}.\qed
\end{proof}


\toukoudel{
\medskip
{\footnotesize
\begin{flushright}
\begin{tabular}{ll}
\begin{tabular}{l}
(Martins)\\
Departamento de Matem{\'a}tica,\\
Instituto de Bioci\^{e}ncias, \\
Letras e Ci\^{e}ncias Exatas,\\
 UNESP - Universidade Estadual Paulista,\\
  C\^{a}mpus de S\~{a}o Jos\'{e} do Rio Preto, \\
SP, Brazil\\
  E-mail: {\tt luciana.martinsO\!\!\!aunesp.br}\\
{}\\
\end{tabular}
&
\begin{tabular}{l}
(Saji)\\
Department of Mathematics,\\
Graduate School of Science, \\
Kobe University, \\
Rokkodai 1-1, Nada, Kobe \\
657-8501, Japan\\
  E-mail: {\tt sajiO\!\!\!amath.kobe-u.ac.jp}\\
\end{tabular}\\
\begin{tabular}{l}
(Teramoto)\\
Institute of Mathematics for Industry, \\
Kyushu University, \\
Motooka 744, Nishi-ku, Fukuoka \\
819-0395, Japan\\
  E-mail: {\tt k-teramotoO\!\!\!aimi.kyushu-u.ac.jp}\\
\end{tabular}
\end{tabular}
\end{flushright}}
}

\begin{thebibliography}{9}
\bibitem{an}
K. Akutagawa and S. Nishikawa,
{\it The Gauss map and spacelike surfaces with prescribed mean
curvature in Minkowski\/ $3$-space},
Tohoku Math. J. {\bf 42} (1990) 67--82.

\bibitem{AGV}
V. I. Arnol'd, S. M. Gusein-Zade and A. N. Varchenko,
{\em Singularities of differentiable maps, Vol. $1$},
Monogr. Math. {\bf 82}, Birkh\"auser Boston, Inc., Boston, MA, 1985.


\bibitem{bw}
J. W. Bruce and J. M. West,
{\it Functions on a crosscap},
Math. Proc. Cambridge Philos. Soc. {\bf 123} (1998), no. 1, 19--39.

\bibitem{fh}
T. Fukui and M. Hasegawa,
{\it Fronts of Whitney umbrella
-a differential geometric approach via blowing up},
J. Singul. {\bf 4} (2012), 35--67.
\bibitem{ft}
T. Fukunaga and M. Takahashi,
{\it Evolutes of fronts in the Euclidean plane},
J. Singul. {\bf 10} (2014), 92--107.

\bibitem{hhnsuy}
M. Hasegawa, A. Honda, K. Naokawa, K. Saji, M. Umehara and K. Yamada,
{\it Intrinsic properties of surfaces with singularities},
Internat. J. Math. {\bf 26} (2015), no. 4, 1540008, 34 pp.

\bibitem{hnuy}
A. Honda, K. Naokawa, M. Umehara and K. Yamada,
{\it Isometric realization of cross caps as formal
power series and its applications},
Hokkaido Math. J. {\bf 48} (2019), 1--44.

\bibitem{hnuy2}
A. Honda, K. Naokawa, M. Umehara and K. Yamada,
{\it Isometric deformations of wave
fronts at non-degenerate singular points},
arXiv:1710.02999.

\bibitem{ishifrontal}
G. Ishikawa,
{\it Recognition Problem of Frontal Singularities},
arXiv:1808.09594.

\bibitem{irrf}
S. Izumiya, M. C. Romero Fuster, M. A. S. Ruas and F. Tari,
{\it Differential geometry from a singularity theory viewpoint},
World Scientific Publishing Co. Pte. Ltd., Hackensack, NJ, 2016.

\bibitem{is}
S. Izumiya and K. Saji,
{\it The mandala of Legendrian dualities for pseudo-spheres
in Lorentz-Minkowski space and ``flat'' spacelike surfaces},
J. Singul. {\bf 2} (2010), 92--127.

\bibitem{ist}
S. Izumiya, K. Saji and M. Takahashi,
{\it Horospherical flat surfaces in hyperbolic\/ $3$-space},
J. Math. Soc. Japan {\bf 62} (2010), no. 3, 789--849.

\bibitem{iste}
S. Izumiya, K. Saji and N. Takeuchi,
{\it Flat surfaces along cuspidal edges},
J. Singul. {\bf 16} (2017), 73--100.

\bibitem{k3}
K. Kenmotsu,
{\it Weierstrass formula for surfaces of prescribed mean curvature},
Math. Ann. {\bf 245} (1979), no. 2, 89--99.

\bibitem{koku}
M. Kokubu,
{\it Application of a unified Kenmotsu-type
formula for surfaces in Euclidean or Lorentzian three-space},
preprint, arXiv:1711.05427.
\bibitem{krsuy}
M.~Kokubu, W.~Rossman, K.~Saji, M.~Umehara, and K.~Yamada,
{\it Singularities of flat fronts in hyperbolic\/ $3$-space},
Pacific J. Math. {\bf 221} (2005), 303--351.

\bibitem{kos}
M. Kossowski,
{\it Realizing a singular first fundamental form as
a nonimmersed surface in Euclidean\/ $3$-space},
J. Geom. {\bf 81} (2004), no. 1-2, 101--113.

\bibitem{m}
M. A. Magid,
{\it Timelike surfaces in Lorentz\/ $3$-space with prescribed
mean curvature and Gauss map},
Hokkaido Math. J. {\bf 19} (1991) 447--464.

\bibitem{MSUY}
L. F. Martins, K. Saji, M. Umehara and K. Yamada,
{\itshape Behavior of Gaussian curvature and
mean curvature near non-degenerate singular
points on wave fronts}, Geometry and Topology of Manifold,
Springer Proc. Math. \& Statistics, 2016, 247--282.

\bibitem{MS}
L. F. Martins and K. Saji,
{\it Geometric invariants of cuspidal edges},
Canad. J. Math. {\bf 68} (2016), no. 2, 445--462.

\bibitem{nuy}
K. Naokawa, M. Umehara and K. Yamada,
{\it Isometric deformations of cuspidal edges},
Tohoku Math. J. (2) {\bf 68} (2016), no. 1, 73--90.

\bibitem{OT} R. Oset Sinha and F. Tari,
{\it Flat geometry of cuspidal edges},
Osaka J. Math. {\bf 55} (2018), no. 3, 393--421.

\bibitem{front}
K. Saji, M. Umehara, and K. Yamada,
{\itshape The geometry of fronts},
Ann. of Math. {\bf 169} (2009), 491--529.

\end{thebibliography}
\end{document}